\documentclass[graybox]{svmult}

\usepackage{graphicx}        % standard LaTeX graphics tool
\usepackage{multicol}        % used for the two-column index
\usepackage[bottom]{footmisc}% places footnotes at page bottom

\usepackage{cite,url}
\usepackage{times}

\usepackage{amssymb}
\usepackage{amsmath}
\usepackage{amsfonts}
\usepackage{arydshln}

\newcommand{\R}{\mathbb{R}}
\newcommand{\N}{\mathbb{N}}
\newcommand{\norm}[1]{\lVert{#1}\rVert}

\begin{document}

\title*{Global Stability Analysis of Nonlinear Sampled-Data Systems using Convex Methods}
 \titlerunning{Analysis of Nonlinear Sampled-Data Systems}

\author{Matthew M. Peet and Alexandre Seuret}
\institute{Matthew M. Peet \at Arizona State University, PO Box 876106, Tempe, AZ 85287, \email{mpeet@asu.edu}
\and Alexandre Seuret \at CNRS, LAAS, 7 avenue du Colonel Roche, 31077 Toulouse, France.
Univ. de Toulouse, LAAS, F-31400 Toulouse, France.
 \email{ aseuret@laas.fr}}

\maketitle

\abstract{In this chapter, we consider the problem of global stability of nonlinear sampled-data systems. Sampled-data systems are a form of hybrid model which arises when discrete measurements and updates are used to control continuous-time plants. In this paper, we use a recently introduced Lyapunov approach to derive stability conditions for both the case of fixed sampling period (synchronous) and the case of a time-varying sampling period (asynchronous). This approach requires the existence of a Lyapunov function which decreases over each sampling interval. To enforce this constraint, we use a form of slack variable which exists over the sampling period, may depend on the sampling period, and allows the Lyapunov function to be temporarily increasing. The resulting conditions are enforced using a new method of convex optimization of polynomial variables known as Sum-of-Squares. We use several numerical examples to illustrate this approach.}

\section{Introduction to the Problem of Stability of Sampled-Data Systems}
\label{sec:1}
Consider an aircraft in combat being remotely piloted by an operator. Directed energy or some other form of electronic warfare is used to deny portions of the electromagnetic spectrum and thus reduce the communication bandwidth between vehicle and operator. The change in bandwidth restricts the rate at which information can be transmitted to the vehicle. The question we ask is what is the minimum rate of transfer of information the aircraft can tolerate before it becomes unstable. This situation is similar to the use of electronic countermeasures on an active radar-guided missile. In both cases, there is a set of continuous-time dynamics representing the dynamics of the controlled system. These continuous dynamics are regulated by continuous real-time feedback using digital sensing and actuation. During normal operation, the controller is updated continuously and so the digitization of the controller does not affect the dynamics of the closed-loop system. When interference occurs, however, the update rate of the controller may be sparse or unpredictable. In this case, the system becomes neither discrete nor continuous, but rather a special type of hybrid system referred to as a \textit{Sampled-Data} system, modeled as
\[
\dot x(t) = f(x(t),x(t_k))\qquad \text{for }t\in [t_k,t_k+T_k],\quad k=1,\cdots,\infty.
\]
where $t_{k+1}=t_k + T_k$ for all $k$ and $T_k$ is the sampling period which may be constant or may depend on $k$. Typically, systems of this form arise when the dynamics depend on external updates - often through the use of a controller so that $f(x(t),x(t_k)) = f^* (x(t),u(t))$ with $u(t) = k(x(t_k))$. The sampling period $T_k$ may be thought of as the time between updates from an external controller. In the scenarios described above, $T_k$ would vary with $k$ - the so-called `asynchronous' case. However, there are certain situations when $T_k$ may not vary from update to update - such as when as when the controller is implemented using using an A/D converter with step-times. We refer to this situation as the `synchronous' case.

Linear Sampled-Data systems have been well-studied in the literature~\cite{ChenFran:1995,FRISEURIC:AUT:04,Fujioka:Autom:2009,
ZhanBranPhi:IEEECSM:2001}, including work on nonlinear systems in~\cite{jury1964stability,zaccarian2003finite}. One popular approach has been to regard the system in continuous-time and use a discontinuous, time-varying delay to represent the hybrid part of the dynamics~\cite{mikheev1988asymptotic}. Unfortunately, this approach has not been completely successful, as the understanding of nonlinear systems with time-varying delay is itself a difficult problem. An alternative approach has been to regard the system in discrete time~\cite{Suh:Auto:2008,OishiFujioka:CDC:2009,
hetel_06}, where the update law is given by the solution map of the continuous-time system over a period $T_k$. For a linear system, this solution map is well-defined using matrix exponentials. For nonlinear systems, it can be approximated over bounded intervals using methods such as the extended Picard iteration~\cite{Peet_2012TAC}. The difficulty with this approach is that the update law is different for every sampling period - meaning that although the approach may work well for a fixed sampling period, for unknown and time-varying sampling period, one has to verify stability over a family of potential solutions. Even in the linear case, this means verification of stability with parametric uncertainty which enters through the exponential. If we have a nonlinear sampled-data system, then even if the vector field is polynomial, the extended Picard iteration yields a polynomial approximation to the solution map - meaning we must test stability of a complicated polynomial vector field with parametric uncertainty - an NP-hard problem.

In this chapter, we consider the use of a new Lyapunov-based approach to stability analysis of sampled-data systems. Specifically, we rely on a Lyapunov result which states that while Lyapunov functions must experience a net decrease over the sampling period, it may be instantaneously increasing~\cite{seuret:auto:2010}. This constraint can be implemented in a Lyapunov context through the use of `spacing functions' - functions which are required to vanish at the endpoints of the sampling period. The main idea behind these functions is that instead of requiring negativity of the Lyapunov function over the entire sampling interval, we only require the sum of the Lyapunov function and the spacing function to be decreasing for all time. The inspiration for this approach came from the previous work on spacing functions for Lyapunov-Krasovskii functionals for stability of time-delay systems in~\cite{peet_2009SICON}. In~\cite{seuret_2013TAC}, we considered the use of this approach for construction of quadratic Lyapunov functions for linear sampled-data systems in both the synchronous and asynchronous cases. The contribution of this chapter is to show how this approach can be extended to prove global stability of nonlinear sampled-data systems.

This chapter is organized as follows. In Section~\ref{sec:2}, we introduce the sampled-data system model and define our concepts of stability. We then give the Lyapunov theorem whose conditions we will test. We then introduce the Sum-of-Squares approach to optimization of polynomial variables, including the use of Positivstellensatz results to enforce local positivity. In Section~\ref{sec:main}, we show how the Sum-of-Squares framework can be used to enforce the stability conditions of Section~\ref{sec:2}. Finally, in Section~\ref{sec:numeric}, we apply the results of the chapter to several cases of nonlinear stability analysis in both the synchronous and asynchronous cases.

\section{Background}
\label{sec:2}

In this section we will first describe the Lyapunov theorem we will use and discuss the conditions that a Lyapunov function must satisfy. Following this, we will briefly discuss the computational framework we will use to enforce the conditions of the Lyapunov theorem. Specifically, we will give background on optimization of polynomials using the Sum-of-Squares methodology (SOS).

\subsection{Sampled-Data Systems}
In this chapter, we consider the stability of solutions of equations of the form
\begin{align}
\dot x(t) &= f(x(t),x(t_k))\qquad \text{for }t\in [t_k,t_k+T_k],\quad k=1,\cdots,\infty.\notag \\
x(t) &= x_0 \label{eqn:system}
\end{align}
where $t_0=0$, $t_{k+1}=t_k + T_k$ for all $k\ge0$ and $T_k$ is the sampling period which may be constant or may depend on $k$. We assume that $T_k$ satisfies some upper bound $T_k \le T_{\max}$ for all $k$. When it exists, we define the continuous-time flow-map $\Gamma(s)$ to be any function which satisfies $\frac{d}{ds}\Gamma(s)z = f(\Gamma(s)z,z)$ for all $s \in [0,T_{\max}]$ and $\Gamma(0)z = z$. If $\Gamma$ exists, then the sampled-data system can be reduced to a discrete-time system as $x_{k+1} = \Gamma(T_k)x_k$. For the linear sampled-data system
\[
\dot x(t) = A_0 x(t)+A_1 x(t_k),
\] we have the explicit solution
\[
\Gamma(s)z = \left(e^{A_0s} + \int_0^se^{A_0(s-\theta)}A_1 d\theta\right)\,z.
\] For a nonlinear system, the solution map $\Gamma$ is difficult to find - although it may be approximated using such methods as Picard iteration.

\begin{definition} We say the Sampled-Data System~\eqref{eqn:system} is globally exponentially stable if there exist positive constants $K, \gamma$ such that for any initial condition $x_0\in \R^n$, and any $x$ satisfying ~\eqref{eqn:system}, we have $\norm{x(t)}\le K\norm{x_0}e^{-\gamma t}$ for all $t\ge 0$.
\end{definition}
For a linear sampled-data system with a uniform bound on $T_k$, global exponential stability is equivalent to $\rho\left(e^{A_0s} + \int_0^se^{A_0(s-\theta)}A_1 d\theta\right)<1$.

\begin{definition} We say the Sampled-Data System~\eqref{eqn:system} is locally exponentially stable on domain $X$ if there exist positive constants $K, \gamma$ such that for any initial condition $x_0\in X$, and any $x$ satisfying ~\eqref{eqn:system}, we have $x(t) \in X$ and $\norm{x(t)}\le K\norm{x_0}e^{-\gamma t}$ for all $t\ge 0$.
\end{definition}

\subsection{A Lyapunov Theorem}

In this theorem, we assume global existence and continuity of solutions.

\textbf{Notation:} For a given solution, $x$, of System~\eqref{eqn:system}, define the function $x_{k}(s) =\Gamma(s)x(t_k)$ for $s \in [0,T_k]$. Associated with $x_k \in \mathcal{C}[0,T_k]$, we denote the supremum norm $\|x_{k}\|_\infty=\sup_{s\in [0,~ T_{\max}]}\norm{x_{k}(s)}$.

\begin{theorem}\label{thm:Lyapunov}\cite{seuret:auto:2010}
Suppose $V: \R^n\rightarrow \R^+$ is continuously differentiable and
\begin{equation}
\mu_1\norm{x}^2 \leq  V(x) \leq \mu_2\norm{x}^2, \quad \text{for all }x \in \R^n.\label{Th11}
\end{equation}
for positive scalars $\mu_1, \mu_2$ with $\mu_1>\mu_2>0$. Then for any positive constants $\alpha, T_{\min}$ and $T_{\max}$ such that $T_k:=t_{k+1}-t_k \in [T_{\min},T_{\max}]$ for all $k \in \N$, the following are equivalent.
\begin{description}
\item[(i)] There exists positive constants $\epsilon$ such that for all solutions $x$ of Equation~\eqref{eqn:system}, and for all $ k \ge 0$,
\[
V(x(t_{k+1}))<e^{-2\alpha T_k}V(x(t_k))-\epsilon \norm{x(t_k)}^2.
\]
\item[(ii)] There exists a positive constant $\delta$ and continuously differentiable functions $Q_k: [0,T_k] \times \mathcal{C}[0,T_k] \rightarrow \mathbb R$ which satisfy the following for all $k \ge 0$.
\begin{equation}
Q_k(T_{k},z)=e^{-2\alpha T_k}Q_k(0,z)\quad \text{ for all } z \in\mathcal{C}[0,t_k]\label{Th12}
\end{equation}
and such that for all solutions of Equation~\eqref{eqn:system}, and for all $t \in [t_k,t_{k+1}]$
\begin{align}
\frac{\textrm d}{\textrm dt} \left[V(x(t))+Q_k(t-t_k,x_{k})\right]+2\alpha V(x(t)) + 2\alpha Q_k(t-t_k,x_{k})<-\delta\|x_{k}\|_\infty.\label{Th13}
\end{align}
\end{description}
Moreover, if either of these statements is satisfied, then System~\eqref{eqn:system} is globally exponentially stable about the
origin with decay rate $\gamma = \alpha$.
\end{theorem}

Note that the function $Q$ is an operator on an infinite-dimensional vector space. Parametrization of a dense subspace of such operators is impossible using digital computation. However, in this paper, we avoid this difficulty by choosing the operator $Q$ to have the form of $Q(s,z) = F(s,z(0),z(T_k),z(s))$. This choice for the structure of $Q$ comes from the proof of Theorem~\ref{thm:Lyapunov} and is non-conservative.

\subsection{Sum-of-Squares Optimization}
Theorem~\ref{thm:Lyapunov} reduces the question of global exponential stability of sampled-data systems to the existence of a Lyapunov function $V$ and a piecewise-continuous `spacing function' $Q$, which jointly satisfy certain pointwise constraints. Specifically, using the structure $Q_k(s,z) = F_k(s,z(0),z(T_k),z(s))$, we require
\begin{align*}
Q_k(T_{k},z)&= F_k(T_k,z(0),z(T_k),z(T_k))\\
& = e^{-2\alpha T_k} F_k(0,z(0),z(T_k),z(0))\\
&= e^{-2\alpha T_k}Q_k(0,z)
\end{align*}
and
\begin{align*}
&\frac{\textrm d}{\textrm dt} \left[V(x(t))+Q_k(t-t_k,x_{k})\right]+2\alpha V(x(t)) + 2\alpha Q_k(t-t_k,x_{k})\\
&=\nabla V(x)^T f(x(t),x(t_k))  + \nabla_x F_k(t-t_k,x(t),x(t_{k+1}),x(t_k))^T f(x)\\
& \qquad + \frac{d}{dt}F_k(t-t_k,x(t),x(t_{k+1}),x(t_k)) +2\alpha V(x(t))\\
&\qquad  + 2\alpha F_k(t-t_k,x(t),x(t_{k+1}),x(t_k))<-\delta\|x(t)\|
\end{align*}
for all $x(t),x(t_{k+1}),x(t_k)\in\R^n$ and $t-t_k \in [0,T_k]$.

To find the functions $F_k$ and $V$ and enforce these constraints, we must optimize functional variables subject to positivity constraints. While this is a very difficult form of optimization, there has been recent progress in this area through the use of sum-of-squares variables. Specifically, we assume the functions $F$ and $V$ are polynomials of bounded degree. The vector space of polynomials of bounded degree is finite dimensional and can be represented using e.g. a set of monomial basis functions. Specifically, if we define the vector of monomials in variables $x$ of degree $d$ or less as $Z_d(x)$, then we can assume that $F_k$ has the form
\[
F_k(s,x,y,z) = c^T Z_d(s,x,y,z)
\]
for some vector $c \in \R^n$. To enforce the positivity constraints, we assume that any positive polynomial, $h$ can be represented as the sum of squared polynomials as
\[
h(x) = \sum_i g_i(x)^2.
\]
While this assumption is somewhat conservative, the conservatism is not significant, as Sum-of-Squares polynomials are known to be dense in the set of positive polynomials. The key advantage to requiring positive polynomials to be sum-of-squares is that the set of sum-of-squares polynomials of bounded degree is precisely parameterized by the set of positive semidefinite matrices with size corresponding to the degree of the polynomials. That is, a polynomial $y(x) = c^T Z_{2d}(x)$ is SOS if and only if
\[
y(x) =c^T Z_{2d}(x)= Z_{d}(x)^T Q Z_d(x)
\]
for some positive semidefinite matrix $Q$ and where recall $Z_d$ is the vector of monomials in variables $x$ of degree $d$ or less. Thus the constraint that $y$ be a SOS polynomial is equivalent to a set of linear equality constraints between the variables $c$ and $Q$, as well as the constraint that $Q \ge 0$.  Thus optimization of SOS polynomials is actually a form of semidefinite programming - for which we have efficient numerical algorithms and implementations - e.g.~\cite{sturm_1999,borchers1999csdp}.

\textbf{Notation:} We denote the constraint that a polynomial $p$ be Sum-of-Squares as $p\in \Sigma_s$.

While polynomials which are SOS will always be globally positive, we occasionally would like to search for polynomials which are only positive on a subset of $\R^n$. This is typically accomplished through the use of SOS multipliers, formalized through certain `Positivstellensatz' results.
\begin{lemma}\label{lem:PS}
Suppose that there exists polynomials $t_i$ and SOS polynomial $s_i \in \Sigma_s$ such that
\[
v(x) = s_0(x) + \sum_i s_i(x)g_i(x)+\sum_j t_i(x) h_j(x)
\]
Then $v(x)\ge 0$ for any $x \in X:=\{x \in \R^n \,:\, g_i(x)\ge 0,\, h_i(x)=0\}$.
\end{lemma}
Thus if we can represent the subset of interest $X$ as a semialgebraic set, then we can enforce positivity on this set using SOS and polynomial variables. Note that $v$ in Lemma~\ref{lem:PS} is not itself a Sum-of-Squares.

As an example, if we wish to enforce positivity of $F_k(s,x,y,z)$ on the interval $s \in [0,T_k]$, then we can search for $SOS$ functions $s_0,s_1$ such that
\[
F_k(s,x,y,z) = s_0(s,x,y,z) + s_1(s,x,y,z)g(s)
\]
where $g(s) = s(T_k-s)$. This function $g$ was chosen because $s\in [0,T_k]$ if any only if $g(s) \ge 0$. Positivstellensatz results~\cite{stengle_1973,schmudgen_1991,putinar_1993} give conditions under which Lemma~\ref{lem:PS} is not conservative.

Polynomial positivity and Sum-of-Squares have been studied for some time. For additional information, we refer the reader to the references~\cite{reznick_2000,parrilo_2000,lasserre_2006,chesi_2005,peet_2009SICON}.

\section{Main Results}\label{sec:main}
Now that we have described our approach, the main results of the paper follow directly. We will describe both the synchronous and asynchronous cases and consider global exponential stability. Note that in the following theorems, we restrict $Q$ to have the structure
\[
Q_k(s,z) = F_k(s,x(0),x(s)).
\]
That is, there is no dependence on $x(T_k)$. This was done in order to be consistent with our approach to linear Sampled-Data systems described in~\cite{seuret_2013TAC} and also to reduce the computational complexity of the stability conditions. This restriction may, however, introduce additional conservatism and should be considered carefully by the user.

\subsection{The Synchronous Case}

We first consider stability in the `synchronous' case - that is, when $T_i=T_j$ for all $i,j>0$. In this case, the updates to the state occur after regular intervals. As we have argued before, this case is often unrealistic. However, there exist certain scenarios where this model is relevant - such as in the case of an A/D converter. Synchronous sampled-data systems are well-represented by conversion to a discrete-time system as the resulting state update law
\[
x_{k+1} = f(x_k)
\]
will not depend on $k$. However, as we mentioned before, derivation and stability analysis of the resulting nonlinear discrete-time system are still difficult problems. For this reason and others, the method we outline in this section will not rely on conversion to discrete-time, but will use SOS programming to perform global exponential stability analysis while retaining the full hybrid model of the dynamics.

\begin{theorem}\label{thm:synch}
Suppose there exist polynomials $V$, $F$, $s_0$, and $s_1$ such that
\begin{align}
& V(x)-\mu_1\norm{x}^2 \in \Sigma_s\label{eqn:cdn1}\\
&\nabla V(z)^T f(z,x)  + \nabla_z F(t,x,z)^T f(z,x) + \frac{d}{dt}F(t,x,z) +2\alpha V(z) + 2\alpha F(t,x,z) \notag \\
&\qquad \qquad =-s_0(t,x,z) - s_1(t,x,z) t(T-t)\label{eqn:cdn2}\\
&F(T,x,y) = e^{-2\alpha T} F(0,x,x)\label{eqn:cdn3}
\end{align}
then if $T_k=T$ for all $k>0$, System~\ref{eqn:system} is globally exponentially stable.
\end{theorem}

\begin{proof}
Using $V$ as given and $Q_k(t,z) = F(t,z(0),z(t))$ for all $k>0$, we first get from Condition~\eqref{eqn:cdn1} that
\[
V(x(t))-\mu_1\norm{x(t)}^2\ge 0
\] and hence
\[
V(x(t))\ge \mu_1\norm{x(t)}^2.
\]
Furthermore, since $V$ is a polynomial, it is upper bounded by some function $\mu_2 \norm{x}^p$ for sufficiently large $\mu_2$ and $p$.

Next, we see that from Condition~\eqref{eqn:cdn3},
\begin{align*}
Q_k(T_{k},z)&=F(T,z(0),z(T))\\
&= e^{-2\alpha T} F(0,z(0),z(0))\\
& =e^{-2\alpha T_k}Q_k(0,z).
\end{align*}

Finally, we have from Condition~\eqref{eqn:cdn2} and Lemma~\ref{lem:PS} that
\begin{align*}
&\frac{\textrm d}{\textrm dt} \left[V(x(t))+Q_k(t-t_k,x_{k})\right]+2\alpha V(x(t)) + 2\alpha Q_k(t-t_k,x_{k})\\
&=\nabla V(x(t))^T f(x(t),x(t_k))  + \nabla_3 F(t,x(t_k),x(t))^T f(x(t),x(t_k))\\
&\qquad \qquad  + \nabla_1 F(t,x(t_k),x(t)) +2\alpha V(x(t)) + 2\alpha F(t,x(t_k),x(t))\\
&\le 0
\end{align*}
for all $s \in [0,T]$. Thus the conditions for exponential stability in Theorem~\ref{thm:Lyapunov} are satisfied. We conclude that System~\eqref{eqn:system} is stable if $T=T_k$ for all $k>0$.
\end{proof}

\subsection{The Asynchronous Case}

In this Subsection, we consider the case when the sampling period is time-varying, yet is known to lie within some interval $[T_{\min}, T_{\max}]$. To illustrate, suppose that during a Denial-of-Service attack the rate of controller updates is reduced, but still does not drop below the rate of one packet per second. Thus implies a maximum sampling period of $T_{\max} = 1s$. However, it is possible and even likely that the duration between most of the updates updates during and after the attack may be significantly less that this $T_{\max}$. Hence, there is also a minimum sample time determined to be either $T_{\min}=0$ or possibly to be the communication delay between controller and system if the application is tele-operation.

To address the problem where we have $T_k \in [T_{\min}. T_{\max}]$, we allow the `spacing function' $F$ to vary with $T_k$. This is allowable since the spacing function is not part of the storage function, $V$.

Note that we do not allow $V$ to be a function of $T_k$. The restriction that $V$ not vary with $k$ is similar to the Quadratic Stability condition for general classes of switched systems. However, which quadratic stability is known to be conservative for general classes of hybrid system, for sampled-data systems it is not known whether quadratic stability is conservative.

\begin{theorem}\label{thm:asynch}
Suppose there exist polynomials $V$, $F$, $s_0$, and $s_1$ such that
\begin{align}
& V(x)-\mu_1\norm{x}^2 \in \Sigma_s\label{eqn:cdn1v2}\\
&\nabla V(z)^T f(z,x)  + \nabla_z F(t,x,z,T)^T f(z,x) + \frac{d}{dt}F(t,x,z,T) +2\alpha V(z) \notag \\
&\qquad \qquad \qquad \qquad  + 2\alpha F(t,x,z,T)\notag \\
&=-s_0(t,x,z,T) - s_1(t,x,z,T) t(T-t) - s_2(t,x,z,T)(T-T_{\min})(T_{\max}-T)\label{eqn:cdn2v2}\\
&F(T,x,y,T) = e^{-2\alpha T_{\max}} F(0,x,x)\label{eqn:cdn3v2}
\end{align}
then if $T_k\in [T_{\min},T_{\max}]$ for all $k>0$, System~\eqref{eqn:system} is globally exponentially stable.
\end{theorem}

\begin{proof}
The proof is similar to the synchronous case. We use $V(x)$ as given and define $Q_k(t,z) = F(t,z(0),z(t),T_k)$ for all $k>0$. From Condition~\eqref{eqn:cdn1v2} we have that
\[
V(x(t))\ge \mu_1\norm{x(t)}^2.
\]
As before, since $V$ is a polynomial, it is upper bounded by some function $\mu_2 \norm{x}^p$ for sufficiently large $\mu_2$ and $p$.

Next, we see that from Condition~\eqref{eqn:cdn3v2},
\begin{align*}
Q_k(T_{k},z)&=F(T_k,z(0),z(T_k),T_k)\\
&= e^{-2\alpha T_k} F(0,z(0),z(0),T_k)\\
& =e^{-2\alpha T_k}Q_k(0,z).
\end{align*}

Finally, we have from Condition~\eqref{eqn:cdn2v2} and Lemma~\ref{lem:PS} that
\begin{align*}
&\frac{\textrm d}{\textrm dt} \left[V(x(t))+Q_k(t-t_k,x_{k})\right]+2\alpha V(x(t)) + 2\alpha Q_k(t-t_k,x_{k})\\
&=\nabla V(x(t))^T f(x(t),x(t_k))  + \nabla_3 F(t,x(t_k),x(t),T_k)^T f(x(t),x(t_k))\\
 &\qquad + \nabla_1 F(t,x(t_k),x(t),T_k) +2\alpha V(x(t)) + 2\alpha F(t,x(t_k),x(t),T_k)\\
&\le 0
\end{align*}
for all $s \in [0,T_k]$ and $T_k \in [T_{\min},T_{\max}]$. Thus the conditions for exponential stability in Theorem~\ref{thm:Lyapunov} are satisfied. We conclude that System~\eqref{eqn:system} is stable if $T_k \in [T_{\min},T_{\max}]$ for all $k>0$.
\end{proof}

\section{Numerical Examples}\label{sec:numeric}

To verify the algorithms described above, we performed  global stability analysis on a set of nonlinear sampled-data systems. In the examples considered here, we let $\alpha \cong 0$, meaning that we are not interested in finding exponential rates of decay. For a study of estimating exponential rates of decay as a function of sampling period for linear systems, we refer to~\cite{seuret_2013TAC}.

\subsection{Example 1:}
For our first set of numerical examples, we consider the class of 1-D nonlinear dynamical systems parameterized by
\[
\dot x(t) = f(x(t))=ax(t)^3 + bx(t)^2+cx(t)
\]
where we assume that the $u(t) = cx(t)$ term represents negative feedback. Without sampling, we know this system is globally stable if and only if $x(t)f(x(t)) > 0$ for all $x \neq 0$. It can be shown that this condition is satisfied if and only if $a<0$ and $c<\frac{b^2}{4a}$. For this example, we initially chose $a = -1$, $b = 2$ and $c=-1.1$. Then, we used a sampled signal for the term $u(t)=c x(t_k)$ to get the following dynamics.
\[
\dot x(t) = -x(t)^3 + 2x(t)^2-1.1x(t_k).
\]
In Table~\ref{tab:results1}, we list the maximum verifiably globally stable sampling period for this system as a function of the polynomial degree used for the variables $V, F$ and $s_1$. These results were obtained using the conditions of Theorem~\ref{thm:synch} implemented using SOSTOOLS coupled with SeDuMi. Due to the known potential for numerical inaccuracies, all solutions were verified a-posteriori using SOS and via simulation. The resulting Lyapunov function and function $F$ are illustrated over a single sampling period in Figure~\ref{fig:1}. The evolution of the system can be seen over multiple sampling periods in Figure~\ref{fig:2}.
\begin{figure}%[b]
\centering
\includegraphics[width=.8\textwidth]{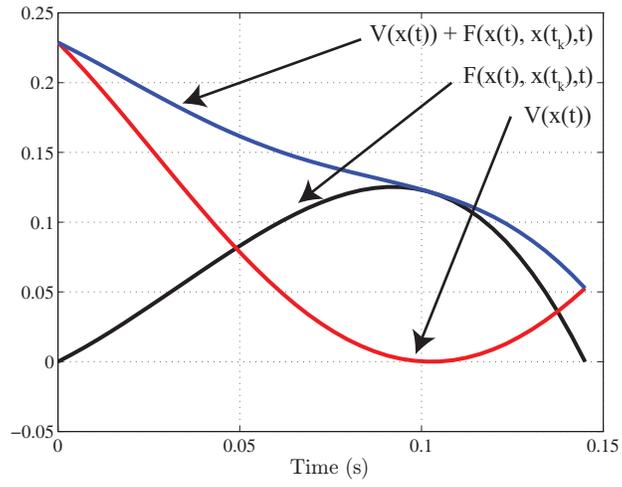}
\caption{Evolution of $V$ and $F$ over one sampling period for Numerical Example 1}
\label{fig:1}       % Give a unique label
\end{figure}
\begin{figure}%[b]
\centering
\includegraphics[width=.8\textwidth]{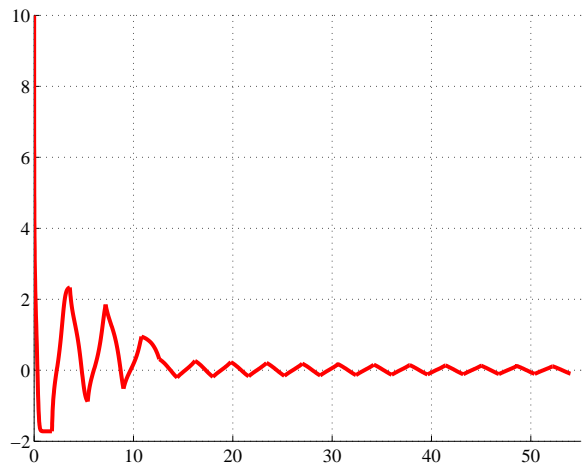}
\caption{Evolution of $x(t)$ over 30 sampling periods for Numerical Example 1 with $T_s = 1.8$}
\label{fig:2}       % Give a unique label
\end{figure}
%\begin{figure}
%\includegraphics[width=\textwidth]{vallado_fig7p10.pdf}
%\caption{Potential Locations of Second Focus for a given $a$}
%\end{figure}

\subsection{Example 2:}
In our second example, we consider a controlled model of a jet engine with dynamics
\begin{align*}
\dot x(t)&=-y(t_k)-\frac{3}{2}x(t)^2 - \frac{1}{2}x(t)^3\\
\dot y(t)&=-y(t)+x(t)
\end{align*}
We consider the case where the negative feedback to the first state is provided using a sampled-data controller. When $T_s=0$, this system is known to be globally stable. Figure~\ref{fig:3} illustrates the trajectories of this system plotted against the level set of one such Lyapunov function for $T_s = .4$.

\begin{figure}%[b]
\centering
\includegraphics[width=.8\textwidth]{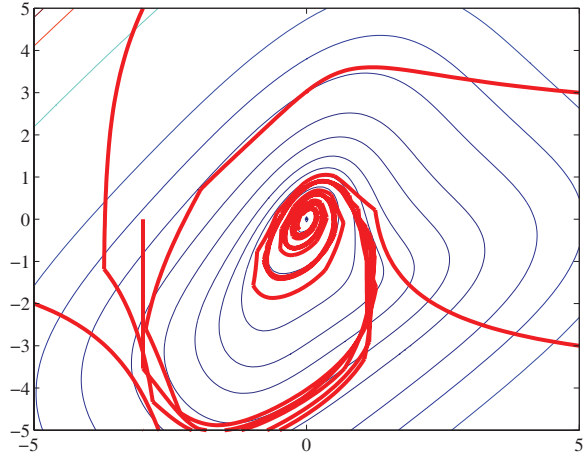}
\caption{Level Sets of a Lyapunov function for Example 2, with multiple trajectories simulated over 30 sampling periods.}
\label{fig:3}       % Give a unique label
\end{figure}

\begin{table}
\begin{center}
\begin{tabular}{|c|c|c|c|c|c|}
\hline
Degree & $N=2$ & $N=4$ & $N=6$ &$N=8$  &  $N=10$ \\
\hline
\hline
\hspace{2mm}Example 1: Maximum Synchronous $T_s$ \hspace{2mm}&\hspace{2mm} $ \emptyset$\hspace{2mm} & \hspace{2mm}$ 0.7901$ \hspace{2mm}&\hspace{2mm} $1.5449$ \hspace{2mm}&\hspace{2mm} $1.8192$ \hspace{2mm}&\hspace{2mm}$1.8411$\hspace{2mm}\\
\hline
\hline
\hspace{2mm}Example 2: Maximum Synchronous $T_s$ \hspace{2mm}&\hspace{2mm} $ \emptyset$ \hspace{2mm}&\hspace{2mm} $.171$ \hspace{2mm}&\hspace{2mm} $.4599$ \hspace{2mm}&\hspace{2mm} N/A \hspace{2mm}&\hspace{2mm} N/A\\
\hline
\hline
\hspace{2mm}Example 3: Maximum Asynchronous $T_s$ \hspace{2mm}&\hspace{2mm} $ \emptyset$ \hspace{2mm}&\hspace{2mm} $.7891$ \hspace{2mm}&\hspace{2mm} $1.542$ \hspace{2mm}&\hspace{2mm} N/A \hspace{2mm}&\hspace{2mm} N/A\\
\hline
\end{tabular}\caption{Maximum allowable sampling period $T_s$
for Examples $1$, $2$, and $3$ with $T_1=0$.}\label{tab:results1}
\end{center}
\end{table}

\subsection{Example 3:}
In this example, we revisit the dynamics of Example 1.
\[
\dot x(t) = f(x(t))=ax(t)^3 + bx(t)^2+cx(t)
\]
However, in this case, we are interest in the case where the sampling period is unknown and time-varying with upper and lower bounds, $T_k \in [T_{\min},T_{\max}]$. Specifically, we choose the lower bound to be $T_{\min}=0$ and determine the maximum upper bound $T_{\max}$ for which stability is retained for all time-varying sampling periods which satisfy $T_k \in [T_{\min},T_{\max}]$. The results are listed in Table~\ref{tab:results1}. As we can see, in this example, allowing the sampling period to vary with time does not significantly affect the maximum sampling period - a surprising result which indicates that using a Lyapunov function $V$ which does not depend on $T_k$ may not be conservative.

\section{Conclusion}
In this chapter, we have studied the question of global stability of nonlinear sampled-data systems in both the synchronous and asynchronous cases. These systems arise through the use of digitized sensing and actuation to control continuous-time dynamics where the controller updates may be irregular. Our approach has been to exploit a new type of slack variable to find Lyapunov functions which experience net decrease over each sampling period, but may be instantaneously increasing at certain points in time. The stability conditions are implemented using a new form of optimization (Sum-of-Squares) which allows us to search for polynomial functions which satisfy pointwise positivity constraints. The result is a convex algorithm which is able to assess global stability of nonlinear vector fields with sampled-data signals in both the asynchronous and the synchronous cases. The effectiveness of the algorithm is demonstrated on several numerical examples.

\begin{acknowledgement}
This work was supported by the National Science Foundation under Grants CMMI 110036 and CMMI 1151018.
\end{acknowledgement}
%

%%%  Do not use Bibtex for your citations
%%%  please use the following:
%\bibliographystyle{IEEEtran}
%\bibliography{bibalex,biblio_MP_peet}

\begin{thebibliography}{10}
\providecommand{\url}[1]{#1}
\csname url@samestyle\endcsname
\providecommand{\newblock}{\relax}
\providecommand{\bibinfo}[2]{#2}
\providecommand{\BIBentrySTDinterwordspacing}{\spaceskip=0pt\relax}
\providecommand{\BIBentryALTinterwordstretchfactor}{4}
\providecommand{\BIBentryALTinterwordspacing}{\spaceskip=\fontdimen2\font plus
\BIBentryALTinterwordstretchfactor\fontdimen3\font minus
  \fontdimen4\font\relax}
\providecommand{\BIBforeignlanguage}[2]{{%
\expandafter\ifx\csname l@#1\endcsname\relax
\typeout{** WARNING: IEEEtran.bst: No hyphenation pattern has been}%
\typeout{** loaded for the language `#1'. Using the pattern for}%
\typeout{** the default language instead.}%
\else
\language=\csname l@#1\endcsname
\fi
#2}}
\providecommand{\BIBdecl}{\relax}
\BIBdecl

\bibitem{ChenFran:1995}
T.~Chen and B.~Francis, \emph{Optimal sampled-data control systems}.\hskip 1em
  plus 0.5em minus 0.4em\relax Berlin, Germany: Springer-Verlag, 1995.

\bibitem{FRISEURIC:AUT:04}
E.~Fridman, A.~Seuret, and J.-P. Richard, ``Robust sampled-data stabilization
  of linear systems: An input delay approach,'' \emph{Automatica}, vol.~40,
  no.~8, pp. 1141--1446, 2004.

\bibitem{Fujioka:Autom:2009}
H.~Fujioka, ``Stability analysis of systems with aperiodic sample-and-hold
  devices,'' \emph{Automatica}, vol.~45, no.~3, pp. 771--775, 2009.

\bibitem{ZhanBranPhi:IEEECSM:2001}
W.~Zhang, M.~Branicky, and S.~Phillips, ``Stability of networked control
  systems,'' \emph{IEEE Control Systems Magazine}, no.~21, 2001.

\bibitem{jury1964stability}
E.~Jury and B.~Lee, ``On the stability of a certain class of nonlinear
  sampled-data systems,'' \emph{Automatic Control, IEEE Transactions on},
  vol.~9, no.~1, pp. 51--61, 1964.

\bibitem{zaccarian2003finite}
L.~Zaccarian, A.~R. Teel, and D.~Ne{\v{s}}i{\'c}, ``On finite gain lp stability
  of nonlinear sampled-data systems,'' \emph{Systems \& control letters},
  vol.~49, no.~3, pp. 201--212, 2003.

\bibitem{mikheev1988asymptotic}
Y.~Mikheev, V.~Sobolev, and E.~Fridman, ``Asymptotic analysis of digital
  control systems,'' \emph{Automation and Remote Control}, vol.~49, no.~9, pp.
  1175--1180, 1988.

\bibitem{Suh:Auto:2008}
Y.~Suh, ``Stability and stabilization of nonuniform sampling systems,''
  \emph{Automatica}, vol.~44, no.~12, pp. 3222--3226, 2008.

\bibitem{OishiFujioka:CDC:2009}
Y.~Oishi and H.~Fujioka, ``Stability and stabilization of aperiodic
  sampled-data control systems: An approach using robust linear matrix
  inequalities,'' in \emph{Joint $48^{th}$ IEEE Conference on Decision and
  Control and $28^{th}$ Chinese Control Conference}, December 16-18 2009.

\bibitem{hetel_06}
L.~Hetel, J.~Daafouz, and C.~Iung, ``Stabilization of arbitrary switched linear
  systems with unknown time-varying delays,'' \emph{IEEE Trans. on Automatic
  Control}, vol.~51, no.~10, pp. 1668--1674, Oct. 2006.

\bibitem{Peet_2012TAC}
M.~Peet and A.~Papachristodoulou, ``A converse sum of squares {Lyapunov} result
  with a degree bound,'' \emph{IEEE Transactions on Automatic Control},
  vol.~57, no.~9, 2012.

\bibitem{seuret:auto:2010}
A.~Seuret, ``A novel stability analysis of linear systems under asynchronous
  samplings,'' \emph{Automatica}, vol.~48, no.~1, pp. 177--182, 2012.

\bibitem{peet_2009SICON}
M.~M. Peet, A.~Papachristodoulou, and S.~Lall, ``Positive forms and stability
  of linear time-delay systems,'' \emph{SIAM Journal on Control and
  Optimization}, vol.~47, no.~6, 2009.

\bibitem{seuret_2013TAC}
A.~Seuret and M.~Peet, ``Stability analysis of sample-data systems using
  sum-of-squares,'' \emph{IEEE Transactions on Automatic Control}, vol.~58,
  no.~6, 2013.

\bibitem{sturm_1999}
J.~F. Sturm, ``Using {SeDuMi} 1.02, a {M}atlab {T}oolbox for optimization over
  symmetric cones,'' \emph{Optimization Methods and Software}, vol. 11-12, pp.
  625--653, 1999.

\bibitem{borchers1999csdp}
B.~Borchers, ``{CSDP}, a {C} library for semidefinite programming,''
  \emph{Optimization Methods and Software}, vol.~11, no. 1-4, pp. 613--623,
  1999.

\bibitem{stengle_1973}
G.~Stengle, ``A nullstellensatz and a positivstellensatz in semialgebraic
  geometry,'' \emph{Mathematische Annalen}, vol. 207, pp. 87--97, 1973.

\bibitem{schmudgen_1991}
C.~Schm\"{u}dgen, ``The {K-moment} problem for compact semi-algebraic sets,''
  \emph{Mathematische Annalen}, vol. 289, no.~2, pp. 203--206, 1991.

\bibitem{putinar_1993}
M.~Putinar, ``Positive polynomials on compact semi-algebraic sets,''
  \emph{Indiana Univ. Math. J.}, vol.~42, no.~3, pp. 969--984, 1993.

\bibitem{reznick_2000}
B.~Reznick, ``Some concrete aspects of {Hilbert's} 17th problem,''
  \emph{Contemporary Mathematics}, vol. 253, pp. 251--272, 2000.

\bibitem{parrilo_2000}
P.~A. Parrilo, ``Structured semidefinite programs and semialgebraic geometry
  methods in robustness and optimization,'' Ph.D. dissertation, California
  Institute of Technology, 2000.

\bibitem{lasserre_2006}
J.~B. Lasserre, ``A sum of squares approximation of nonnegative polynomials,''
  \emph{SIAM Journal of Optimization}, vol.~16, no.~3, pp. 751--765, 2006.

\bibitem{chesi_2005}
G.~Chesi, A.~Garulli, A.~Tesi, and A.~Vincino, ``Polynomially
  parameter-dependent lyapunov functions for robust stability of polytopic
  systems: An lmi approach,'' vol.~50, no.~3, pp. 365--370, 2005.

\end{thebibliography}
%\begin{thebibliography}{99.}
%
%\bibitem{ref1} xxxxx
%
%SEE referenc.tex and authsamp.pdf FILES FOR EXAMPLES OF CITATION STYLES
%
%
%\end{thebibliography}

\end{document}